\documentclass[oneside,english,reqno,11pt]{amsart}
\usepackage{geometry}
\usepackage{marvosym}
\usepackage{amsthm}
\usepackage{amssymb}
\usepackage{setspace}
\usepackage{tikz}
\usepackage{hyperref}
\usepackage{enumitem}
\usetikzlibrary{arrows}
\tikzstyle{vertex}=[draw, circle, scale=.7, minimum size=.8cm]
\makeatletter
\numberwithin{equation}{section}
\numberwithin{figure}{section}
\theoremstyle{plain}
\newtheorem{theorem}{\protect\theoremname}
  \newtheorem{cor}[theorem]{\protect\corollaryname}
  \newtheorem{lemma}[theorem]{\protect\lemmaname}
  \theoremstyle{definition}

  \theoremstyle{plain}
  \theoremstyle{remark}

  \numberwithin{theorem}{section}

\newcommand{\OR}[2]{\operatorname{OR}_{#1}({#2})}

\newcommand{\R}[2]{\operatorname{R}_{#1}(#2)}
\newcommand{\tow}[2]{\operatorname{tow}_{#1}(#2)}

\makeatother

\usepackage{babel}
  \providecommand{\claimname}{Claim}
  \providecommand{\definitionname}{Definition}
  \providecommand{\lemmaname}{Lemma}
  \providecommand{\remarkname}{Remark}
  \providecommand{\corollaryname}{Corollary}
  \providecommand{\propname}{Proposition}
  \providecommand{\examplename}{Example}
  \providecommand{\obsname}{Observation}
\providecommand{\theoremname}{Theorem}

\begin{document}

\title{Ordered Ramsey Numbers of Loose Paths and Matchings}

\author{Christopher Cox$^1$ \and Derrick Stolee$^1$}

\date{\today}
\maketitle

\begin{abstract}
For a $k$-uniform hypergraph $G$ with vertex set $\{1,\ldots,n\}$, the ordered Ramsey number $\OR{t}{G}$ is the least integer $N$ such that every $t$-coloring of the edges of the complete $k$-uniform graph on vertex set $\{1,\ldots,N\}$ contains a monochromatic copy of $G$ whose vertices follow the prescribed order. 
Due to this added order restriction, the ordered Ramsey numbers can be much larger than the usual graph Ramsey numbers. 
We  determine that the ordered Ramsey numbers of loose paths under a monotone order grows as a tower of height one less than the maximum degree. 
We also extend theorems of Conlon, Fox, Lee, and Sudakov [Ordered Ramsey numbers, arXiv:1410.5292] on the ordered Ramsey numbers of 2-uniform matchings to provide upper bounds on the ordered Ramsey number of $k$-uniform matchings under certain orderings.
\end{abstract}

\spacing{1.0}

\footnotetext[1]{Iowa State University, Ames, IA, USA. \texttt{\{cocox,dstolee\}@iastate.edu}.}

\section{Introduction}

Ramsey theory is a fundamental topic in extremal graph theory.
The Ramsey number  $\R{t}{n}$ is the minimum $N$ such that every $t$-coloring of the edges of the complete graph of order $N$ contains a monochromatic clique of order $n$.
The number $\R{t}{n}$ can also be defined as the maximum $N$ such that there exists a $t$-coloring of $K_{N-1}$ that avoids monochromatic copies of the graph $K_n$.
This concept naturally generalizes to avoiding monochromatic copies of any $k$-uniform hypergraph $G$, defining the graph Ramsey number $\R{t}{G}$, leading to a large number of available questions.
The asymptotic growth of $\R{t}{G}$ varies significantly, and depends on several properties of $G$, such as maximum degree~\cite{CRST} or degeneracy~\cite{FS}.

A recent variation, called \emph{ordered Ramsey theory}, has received significant attention~\cite{BK, CGKVV, CFLS,  FPSS, MSW, MS}.
In this variation, we again look for $t$-colorings of the complete graph that avoid monochromatic copies of a graph $G$, except that the \emph{order} of the vertices of $G$ in this monochromatic copy are very important.
This modification relaxes some of the constraints on the coloring, so the ordered Ramsey numbers can be much larger than the usual graph Ramsey number, but is still bounded from above by the Ramsey number $\R{t}{n}$ where $n$ is the number of vertices in $G$.
If $G$ is a 2-uniform path under the standard ordering, then the 2-color ordered Ramsey number of $G$ is equal to the bound of the Erd\H{o}s-Szekeres Theorem~\cite{ES} (see \cite{CP,MSW}).
If $G$ is a tight 3-uniform path under the standard ordering, then the 2-color ordered Ramsey number of $G$ is equal to the bound of the \emph{happy ending problem} (see \cite{FPSS}).
Due to these connections, much of the previous work has focused on the ordered Ramsey number of tight $k$-uniform paths under the standard ordering~\cite{FPSS,MSW,MS}, but others considered 2-uniform matchings with an arbitrary ordering~\cite{BK,CFLS}.
We extend these investigations by determining strong bounds on the ordered Ramsey number of loose $k$-uniform paths and $k$-uniform matchings, using an arbitrary number of colors. 

An \emph{ordered $k$-uniform hypergraph} is a hypergraph $G$ where the edge set $E(G)$ contains $k$-sets of vertices, and the vertex set $V(G)$ is totally ordered.
An ordered hypergraph $G$ is \emph{contained} in an ordered hypergraph $H$ if there is an injective, order-preserving map from the vertices of $G$ to the vertices of $H$ such that edges of $G$ map to edges of $H$.
Let $K_N^k$ be the complete $k$-uniform hypergraph on the vertex set $\{1,\dots,N\}$ and let $c : E(K_N^k) \to \{1,\dots,t\}$ be a $t$-coloring of the edges in $K_N^k$.
The $i$-colored subgraph of $K_N^k$ is the ordered hypergraph given by the edges in $c^{-1}(i)$.
For ordered $k$-uniform hypergraphs $G_1, \dots, G_t$, the \emph{ordered Ramsey number} $\OR{}{G_1,\dots,G_t}$ is the minimum $N$ such that for every $t$-coloring of $K_N^k$ there is some color $i$ such that the $i$-colored subgraph contains $G_i$.
This number is necessarily defined and finite, since there exists an $n$ such that each $G_i$ is a subgraph of $K_n^k$ and hence $\OR{}{G_1,\dots,G_t} \leq \R{t}{n}$.
If $G_1 = \dots = G_t = G$, then we denote $\OR{}{G_1,\dots,G_t}$ as $\OR{t}{G}$ and refer to this as the \emph{diagonal case}; otherwise it is the \emph{off-diagonal case}.

For positive integers $k, \ell, e$ such that $k > \ell$, the \emph{$(k,\ell)$-path on $e$ edges}, denoted $P_e^{k,\ell}$, is the $k$-uniform ordered hypergraph on $e(k-\ell)+\ell$ vertices and $e$ totally-ordered edges $A_1,A_2,\dots,A_e$ where two consecutive edges $A_i, A_{i+1}$ intersect exactly on the maximum $\ell$ vertices in $A_i$ and the minimum $\ell$ vertices in $A_{i+1}$.
The path $P_e^{k,k-1}$ is called the \emph{tight $k$-uniform path} and otherwise $P_e^{k,\ell}$ is a \emph{loose path}.
For $\ell = 0$, we can extend the definition of $P_e^{k,\ell}$ by requiring that two consecutive edges $A_i, A_{i+1}$ satisfy $\max A_i < \min A_{i+1}$, and hence the edges are disjoint, forming a \emph{matching}.
Note that when $k = 2$ the only possibilities are a tight path or a matching.
We will primarily use the ordering given by this definition, and we will specify the special cases when we will consider a possibly different ordering on $P_e^{k,\ell}$.

Define the \emph{intersection number}, $i(k,\ell)$, to be the maximum degree of a vertex in $P_e^{k,\ell}$ for all $e \geq k$.
Observe that if $\ell > 0$, then $i(k,\ell)$ is the unique integer $m \geq 2$ that satisfies
\[
\frac{m-2}{m-1} < \frac{\ell}{k} \leq \frac{m-1}{m}.
\]

The tight paths $P_e^{k,k-1}$ have been investigated thoroughly.
For 2-uniform tight paths, the ordered Ramsey number $\OR{t}{P_e^{2,1}}$ is determined by Choudum and Ponnusamy~\cite{CP}, and the off-diagonal case of the number $\OR{}{P_{e_1}^{2,1},\dots,P_{e_t}^{2,1}}$ is demonstrated in full generality by Milans, Stolee, and West~\cite{MSW}.
Fox, Pach, Sudakov, and Suk~\cite{FPSS} determined the growth of $\OR{t}{P_e^{3,2}}$ to be doubly-exponential, and Moshkovitz and Shapira~\cite{MS} found that $\OR{t}{P_e^{k,k-1}}$ grows as a tower of height $k-1$.
In fact, Moshkovitz and Shapira determine $\OR{t}{P_e^{k,k-1}}$ exactly in terms of high-dimensional integer partitions.
In Section~\ref{paths}, we use a version of this theorem using partially-ordered sets (posets), due to Milans, Stolee, and West~\cite{MSW}, in order to prove the following relationship between ordered Ramsey numbers of tight and loose paths.

\begin{theorem}\label{thm:main}
For $k > \ell \geq 1$, $i = i(k,\ell)$, and positive integers $e_1,\dots,e_t$,
\[
	\OR{}{P_{e_1}^{k,\ell}, \dots, P_{e_t}^{k,\ell}} = (k-\ell)\OR{}{P_{e_1}^{i,i-1}, \dots, P_{e_t}^{i,i-1}} + \ell - (k-\ell)(i-1).
\]
\end{theorem}
Therefore, the asymptotic growth of $\OR{t}{P_e^{k,\ell}}$ is a tower of height $i(k,\ell)-1$.
In particular, when $i(k,\ell) = 2$ we can use the exact theorem for 2-uniform tight paths to exactly determine the ordered Ramsey number.

\begin{cor}\label{i=2}
For $0 < 2\ell \leq k$ and positive integers $e_1,\dots,e_t$,
\[
	\OR{}{P_{e_1}^{k,\ell}, \dots, P_{e_t}^{k,\ell}} = (k-\ell)\prod_{i=1}^t e_i + \ell.
\]
\end{cor}

Conlon, Fox, Lee, and Sudakov~\cite{CFLS} and Balko, Cibulka, Kr\'{a}l, and Kyn\u{c}l~\cite{BK} independently investigated how the ordered Ramsey number $\OR{t}{G}$ differs among orderings of a 2-uniform graph $G$. In particular, they investigated upper bounds of $\OR{t}{M}$ for a 2-uniform matching $M$, and found that these upper bounds are nearly sharp.
In Section~\ref{matching}, we extend the methods in these papers to attain upper bounds on the ordered Ramsey numbers of $k$-uniform matchings under certain ``controlled'' orderings.
We present an upper bound on the $t$-color ordered Ramsey number $\OR{t}{P_e^{2,1}}$ for an arbitrarily-ordered copy of $P_e^{2,1}$ that nearly matches the upper bound on $\OR{t}{M}$ for a 2-uniform matching $M$, extending work of Cibulka, Gao, Kr\v{c}\'al, Valla, and Valtr~\cite{CGKVV} on two colors.
Several conjectures and open problems are presented in Section~\ref{future}.

\subsection{Notation}

We follow standard notation from~\cite{West}.
For an (ordered) hypergraph $G$, we use $V(G)$ as the vertex set of $G$, $E(G)$ as the edge set of $G$, $|G|$ as the number of edges in $G$, and $k$ will always denote the size of an edge in $G$.
For integers $m \leq n$, let $[n] = \{1,\dots,n\}$, $[m,n] = \{m, m+1, \dots, n-1, n\}$, and let ${[n]\choose m}$ denote the set of $m$-element subsets of $[n]$.
For $k \geq 2$, the complete $k$-uniform (ordered) hypergraph with vertex set $[N]$ is denoted $K_N^k$.
The 2-uniform case is special, so $K_N$ denotes $K_N^2$.

We use $\lg n = \log_2 n$. We always use $e$ the number of edges in a graph and never as the base of the natural logarithm.
The \emph{tower function of height $t$}, denoted by $\tow{t}{n}$, is 
\[
	\tow0n = n, \quad\text{and}\quad \tow{t}{n} = 2^{\tow{t-1}{n}}\text{ for $t \geq 1$.}
\]
We use $\subseteq$ to denote any partial order, including the containment order.
We use $\leq$ to denote a total order, including a linear extension of a partial order.
A list $(x_1,\ldots,x_n)$ is \emph{descent-free} if $x_i\not\supseteq x_{i+1}$ for all $i\in [n-1]$. 
Note that any sublist of a linear extension is descent-free.

\section{Ordered Ramsey Numbers of Loose Paths}\label{paths}

To study the ordered Ramsey number of loose paths, we first review the previous results on the ordered Ramsey number of tight paths.
For a poset $P = (P,\subseteq)$, a \emph{down-set} is a set $S \subseteq P$ such that if $y \in S$ and $x \subseteq y$, then $x \in S$.
For a set $A \subseteq P$, let $D(A)$ be the minimal down-set containing $A$; observe that $D$ forms a bijection between antichains and down-sets of $P$.
The poset $J(P)$ consists of all down-sets in $P$, ordered by containment.

Let $m, e_1,\dots,e_t$ be positive integers and $m \geq 1$.
Define the poset $Q_m(e_1,\dots,e_t)$ iteratively as follows: let $Q_1(e_1,\dots,e_t)$ be a disjoint union of $t$ chains of size $e_1-1, \dots, e_t-1$, and $Q_{m+1}(e_1,\dots,e_t) = J(Q_m(e_1,\dots,e_t))$.
The size of $Q_k(e_1,\ldots,e_t)$ is equal to the largest $N$ such that we can $t$-color $K_N^k$ while avoiding ordered copies of $P_{e_1}^{k,k-1}, \dots, P_{e_t}^{k,k-1}$.

\begin{theorem}[Moshkovitz and Shapira~\cite{MS}; Milans, Stolee, and West~\cite{MSW}]\label{mswthm}
Let $k, e_1,\dots,e_t$ be positive integers and $k \geq 2$.
Then,
\[
	\OR{}{P_{e_1}^{k,k-1},\dots,P_{e_t}^{k,k-1}} = |Q_k(e_1,\dots,e_t)| + 1.
\]
\end{theorem}

We extend this result to loose paths by referring to the same poset definitions.
In particular, the most important parameter affecting the asymptotic growth of $\OR{t}{P_e^{k,\ell}}$ is $i(k,\ell)$, and the value $k$ contributes only to the leading constant.

\begin{theorem}\label{mainthmq}
If $k > \ell \geq 1$ and $e_1,\ldots,e_t$ are positive integers, then 
\[
\OR{}{P_{e_1}^{k,\ell},\ldots,P_{e_t}^{k,\ell}}=(k-\ell)|Q_{i(k,\ell)}(e_1,\dots,e_t)|+\ell-(k-\ell)(i(k,\ell)-2).
\]
\end{theorem}

\begin{proof}
Note that if $e_i = 1$ for any $i$, then any $t$-coloring avoiding an $i$-colored copy of $P_1^{k,\ell}$ will not use the color $i$; hence $e_i$ can be removed from the list and we can consider $t-1$ coloring.
Also note that $Q_1(e_1,\dots,e_t)$ equals $Q_1(e_1',\dots,e_{t'}')$ where $e_1',\dots,e_{t'}'$ is the list of integers $e_j \geq 2$ for $j \in [t]$.

Let $i = i(k,\ell)$ and $\ell'=\ell-(k-\ell)(i-2)$.
For $m \in [i]$, let $Q_m = Q_m(e_1,\dots,e_t)$.
Let $C_1 \cup \cdots \cup C_t$ be a partition of $Q_1$ into a disjoint union of $t$ chains such that each $C_j$ contains $e_j-1$ elements.

Let $A_1,\ldots,A_{k-\ell}$ be copies of $Q_{i}$ and let $\pi:\bigcup_{j=1}^{k-\ell}A_j\to Q_{i}$ be the natural projection map. 
Also, let $L$ be a chain of size $\ell'-1$. 
Define $Q_{i}^*=A_1\cup\cdots\cup A_{k-\ell}\cup L$ to be a poset with the relation between two distinct elements $x,y \in Q_i^*$ defined as:
\begin{itemize}
\item If $x,y\in L$, keep the same relation as in $L$.
\item If $x\in A_j$ and $y\in L$, let $x<y$.
\item If $x\in A_j$ and $y\in A_{j'}$, where $\pi(x) \neq \pi(y)$, provide $x$ and $y$ with the same relationship as $\pi(x)$ and $\pi(y)$.
\item If $x\in A_j$ and $y\in A_{j'}$, where $\pi(x) = \pi(y)$, let $x\leq y$ if $j\leq j'$.
\end{itemize}

We show that $\OR{}{P_{e_1}^{k,\ell},\ldots,P_{e_t}^{k,\ell}}=|Q_{i}^*|+1$.

\textit{Lower Bound.} Fix a linear extension of $Q_i^*$. 
We consider $\pi$ to be a a projection from $Q_i^*\setminus L\to Q_i$. 
For a list $(x_1,\ldots,x_n)$ in $Q_i^*\setminus L$, we extend $\pi$ so that $\pi(x_1,\ldots,x_n)=(\pi(x_1),\ldots,\pi(x_n))$. 
Further, given a list $(x_1,\ldots,x_n)$ in $Q_i^*$, we define the \emph{reduction} of the list to be $r(x_1,\ldots,x_n)=(x_1,x_{(k-\ell)+1},\ldots,x_{s(k-\ell)+1})$ where $s$ is the largest integer such that $s(k-\ell)+1\leq n$. 

Notice first that $r(x_1,\ldots,x_{s(k-\ell)+\ell})=(x_1,x_{(k-\ell)+1},\ldots,x_{(k-\ell)(s+i-2)+1})$ and that $\ell'=(s(k-\ell)+\ell)-(k-\ell)(s+i-2)$. Hence, if $(x_1,\ldots,x_{s(k-\ell)+\ell})$ is a sublist of the linear extension of $Q_i^*$, then $r(x_1,\ldots,x_{s(k-\ell)+\ell})$ is a descent-free list in $Q_i^*\setminus L$. 

Note that in this linear extension of $Q_i^*$, if $x\in A_j$ and $y\in A_{j+1}$ with $\pi(x)=\pi(y)$, then there is no $z\in Q_i^*$ such that $x<z<y$. 
Therefore, if $(x_1,\ldots,x_{s(k-\ell)+\ell})$ is a descent-free list in $Q_i^*$, then not only is $r(x_1,\ldots,x_{s(k-\ell)+\ell})$ a descent-free list in $Q_i^*\setminus L$, but $\pi(r(x_1,\ldots,x_{s(k-\ell)+\ell}))$ is a descent-free list with no repetition in $Q_i$. 

Now, consider $2 \leq m \leq i$ and let $x,y \in Q_m$ with $x\nsupseteq y$. 
Let $f_m(x,y)$ be some element of the set $y \setminus x$ inside of $Q_{m-1}$. 
Further, we extend $f_m$ so that if $(x_1,\ldots,x_n)$ is a descent-free list in $Q_m$, then $f_m(x_1,\ldots,x_n)=(f_m(x_1,x_2), \ldots, f_m(x_{n-1},x_n))$. 
If $x \nsupseteq y$ and $y\nsupseteq z$, then $f_m(x,y)\in y\setminus x$ and $f_m(y,z)\in z\setminus y$, so $f_m(x,y)\nsupseteq f_m(y,z)$. 
Hence, if $(x_1,\ldots,x_n)$ is a descent-free list in $Q_m$, then $f_m(x_1,\ldots,x_n)$ is a descent-free list of length $n-1$ in $Q_{m-1}$. 
For a decent-free list $(x_1,\dots,x_n)$ in $Q_i$, define $f^{(0)}(x_1,\dots,x_n) = f_i(x_1,\dots,x_n)$ and $f^{(h)}(x_1,\ldots,x_n)=f_{i-h}(f^{(h-1)}(x_1,\ldots,x_n))$. 
Observe that if  $(x_1,\dots,x_n)$ is a descent-free list of length $n$ in $Q_i$, then $f^{(h)}(x_1,\ldots,x_n)$ is a descent-free list of length $n-h$ in $Q_{i-h}$. 

For a descent-free list $(x_1,\ldots,x_k)$ in $Q_{i}^*$, let $(y_1,\ldots,y_{i})$ be defined as
\[
(y_1,\dots,y_i)=(\pi(x_1),\pi(x_{(k-\ell)+1}),\ldots,\pi(x_{(k-\ell)(i-1)+1}))=\pi(r(x_1,\ldots,x_k)).
\]
Observe that $(y_1,\ldots,y_{i})$ is a descent-free list in $Q_{i}$, so $f^{(i-1)}(y_1,\ldots,y_{i})$ is an element in $Q_1$. 

For $N = |Q_i^*|$, define a $t$-coloring $c$ on $E(K_N^k)$ as $c(x_1,\ldots,x_k)=j$ whenever $f^{(i-1)}(y_1,\ldots,y_{i})\in C_j$, for $(y_1,\dots,y_i) = \pi(r(x_1,\dots,x_k))$.
We now demonstrate that the coloring $c$ avoids a $j$-colored $P_{e_j}^{k,\ell}$ for all colors $j \in [t]$.

Suppose that $(x_1,\ldots,x_{s(k-\ell)+\ell})$ is the vertex set of a $j$-colored copy of $P_s^{k,\ell}$ for some $s \geq 1$. 
Let 
\[
(y_1,\ldots,y_{s+i-1})=(\pi(x_1),\ldots,\pi(x_{(k-\ell)(s+i-2)+1}))=\pi(r(x_1,\ldots,x_{s(k-\ell)+\ell})).
\]
Notice that $(x_{(k-\ell)(r-1)+1},\dots,x_{(k-\ell)(r-1)+k})$ is an edge of $P_s^{k,\ell}$ for $r\in\{1,\ldots,s\}$, and \[(y_r,y_{r+1},\ldots,y_{r+i-1}) = \pi(r(x_{(k-\ell)(r-1)+1},\dots,x_{(k-\ell)(r-1)+k})).\]
Thus, $f^{(i-1)}(y_r,y_{r+1},\dots,y_{r+i-1})$ is an element of the chain $C_j$, so $f^{(i-1)}(y_1,\ldots,y_{s+i-1})$ is a descent-free list of length $s$ in $C_j$. 
Because a descent-free list in a chain must be strictly increasing, $s\leq |C_j| = e_j-1$. 
Thus, $c$ avoids $P_{e_j}^{k,\ell}$ in color $j$ for each $j\in [t]$.

\textit{Upper Bound.}
Let $c$ be a $t$-coloring of $E(K_N^k)$ that avoids $P_{e_j}^{k,\ell}$ in color $j$ for all $j\in[t]$. 
We will show that $N \leq |Q_{i}^*|$.

For $Y\subseteq [N]$ with $|Y|=h > k-\ell$, let $Y^+$ denote the $h-(k-\ell)$ largest elements of $Y$ and $Y^-$ denote the $h-(k-\ell)$ smallest elements of $Y$. 
We will begin by iteratively defining a function $g_m:{[N]\choose k-(m-1)(k-\ell)}\to Q_m$ for $m \in [i]$ with the property that for all $Y\in{[N]\choose k-(m-2)(k-\ell)}$, $g_m(Y^-)\nsupseteq g_m(Y^+)$. 

We start with the case $m=1$.
Suppose that $X\in{[N]\choose k}$ with $c(X)=j$. 
Let $h$ be the largest integer such that there is an $j$-colored $P_{h}^{k,\ell}$ that has $X$ as its maximum edge. 
Because $c$ avoids $P_{e_j}^{k,\ell}$ in color $j$, $h\leq e_j-1$. 
Supposing that $x_1\subset\cdots\subset x_{e_j-1}$ are the elements of $C_j$ in $Q_1$, let $g_1(X)=x_h$. 
For $Y\in{[N]\choose 2k-\ell}$, if $c(Y^-)\neq c(Y^+)$, then $g_1(Y^-)$ and $g_1(Y^+)$ are in different chains of $Q_1$, so they are not comparable. 
If $c(Y^-)=c(Y^+)$, then $g_1(Y^+)\supseteq g_1(Y^-)$ because $Y^-$ and $Y^+$ form a $P_2^{k,\ell}$ in color $c(Y^-)=c(Y^+)$. 
Therefore $g_1(Y^-)\nsupseteq g_1(Y^+)$.

Let $1<m\leq i$, and for $X\in{[N]\choose k-(m-1)(k-\ell)}$, define $g_m(X)=D(\{g_{m-1}(Y):Y^+=X\})$. 
Because $Q_m=J(Q_{m-1})$, $g_j(X)\in Q_j$. 
Suppose that $Y\in{[N]\choose k-(m-2)(k-\ell)}$ and note that $g_{m-1}(Y)\in g_m(Y^+)$. 
If also $g_{m-1}(Y)\in g_m(Y^-)$, then there is some $Z\in{[N]\choose k-(m-2)(k-\ell)}$ such that $Z^+=Y^-$ and $g_{m-1}(Y)\subseteq g_{m-1}(Z)$. 
For $W=Y\cup Z$, it holds that $W^-=Z$ and $W^+=Y$, so $g_{m-1}(W^-)\supseteq g_{m-1}(W^+)$; a contradiction. 
Therefore, $g_{m-1}(Y)\in g_m(Y^+)\setminus g_m(Y^-)$, so $g_m(Y^-)\nsupseteq g_m(Y^+)$.

Now that $g_{i}$ is defined, and $g_i$ maps ${[N]\choose \ell'}$ to $Q_{i}$, we construct a function $\phi : \{\ell',\ldots,N\}\to Q_{i}$. 
For $\ell' \leq x\leq n$, let $\phi(x) = g_{i}(\{x-\ell'+1,\ldots,x\})$. 
We claim that for any $R \in Q_{i}$, $|\phi^{-1}(R)|\leq k-\ell$. 
If $\ell'\leq x_1<\cdots<x_{k-\ell+1}\leq n$, then $\phi(x_1)=\cdots=\phi(x_{k-\ell+1})$. 
Let $W=\{x_{k-\ell+1}-\ell'+1,\ldots,x_{k-\ell+1}\}$ and $Y=\{x_1-\ell'+1,\ldots,x_1\}$.
Since $\phi(x_1)=\phi(x_{k-\ell+1})$ by assumption, we have $g_{i}(Y)=g_{i}(W)$. 
In particular, $g_{i}(Y)\supseteq g_{i}(W)$ as elements in $Q_{i}$. 
Realizing that $x_{k-\ell-\ell'+1}<\min W$, let $X = Y \cup \{x_1,\ldots, x_{k-\ell-\ell'+1}\} \cup W$. 
Note that $|X|=\ell'+k-\ell$ and that $X^-=Y$ while $X^+=W$. 
However,  $X\in{[N]\choose \ell'+k-\ell}$ and $g_{i}(X^-)\nsupseteq g_{i}(X^+)$, a contradiction.

Since $|\phi^{-1}(R)|\leq k-\ell$ for all $R\in Q_{i}$, $N-\ell' + 1\leq (k-\ell)|Q_{i}| = |Q_i^*| - (\ell'-1)$, so $N \leq |Q_{i}^*|$. 
\end{proof}

Theorem~\ref{thm:main} follows from Theorems~\ref{mswthm} and \ref{mainthmq}.
Corollary~\ref{i=2} follows from Theorem~\ref{mainthmq} after observing that $|Q_2(e_1,\dots,e_t)|=\prod_{j=1}^te_j$ because we can select a down-set of $Q_1(e_1,\dots,e_t)$ by selecting at most one element from each chain to be a maximal element of the down-set.

For $m\geq 3$, the value of $|Q_m(e_1,\dots,e_t)|$ is not known exactly, but note that $|Q_3(e_1,\dots,e_t)|$ is the number of antichains in $Q_2(e_1,\dots,e_t)$. 
When $e_1 = \cdots = e_t = 2$, the poset $Q_2(e_1,\dots,e_t)$ is the $t$-dimensional boolean lattice, denoted $2^{[t]}$, and counting the number of antichains in $2^{[t]}$ is already a famous and difficult problem known as Dedekind's problem.
Thus, we will use the bounds of Moshkovitz and Shapira on $\OR{t}{P_e^{k,k-1}}$~\cite[Corollary 3]{MS} to find the following corollary.

\begin{cor}
For $e\geq 2$, $k< 2\ell<2k$, and $\ell' = \ell-(k-\ell)(i(k,\ell)-1)$,
\[
(k-\ell)\tow{i(k,\ell)-2}{e^{t-1}/2\sqrt{t}}+\ell' \leq \OR{t}{P_e^{k,\ell}} \leq(k-\ell)\tow{i(k,\ell)-2}{2e^{t-1}}+\ell'.
\]
\end{cor}

In \cite{GG}, Gerencs\'{e}r and Gy\'{a}rf\'{a}s showed that for $n\geq m\geq 1$,
\[
\R{}{P_n^{2,1},P_m^{2,1}}=n+\left\lfloor{m\over 2}\right\rfloor+2.
\]
Comparatively, $\OR{}{P_n^{2,1},P_m^{2,1}}=nm+1$, which shows a large discrepancy between the ordered and unordered variants of the Ramsey number in just the $2$-uniform case. 
It should, however, be noted that over all orderings of a $(k,\ell)$-path, the standard ordering on $P_e^{k,\ell}$ does not necessarily minimize the ordered Ramsey number. 
For example, it is easy to observe that there exists an ordering of $P_2^{k,k-1}$ such that $\OR{t}{P_2^{k,k-1}}\leq k+t-1$.

The proof of Theorem~\ref{thm:main} using Theorem~\ref{mainthmq} is valuable because it shows a direct connection between the poset $Q_i(e_1,\dots,e_t)$ and the ordered Ramsey number $\OR{}{P_{e_1}^{k,\ell},\dots,P_{e_t}^{k,\ell}}$ and the best asymptotic bounds on the ordered Ramsey numbers come from this poset perspective.
However, Theorem~\ref{thm:main} can be proven directly by translating $t$-colorings that avoid $(k,\ell)$-paths with $t$-colorings that avoid tight $i$-uniform paths.

\begin{proof}[Direct Proof of Theorem~\ref{thm:main}]
Recall from the proof of Theorem~\ref{mainthmq} the definitions of $i$ and $\ell'$. Let $N = \OR{}{P_{e_1}^{i,i-1}, \dots, P_{e_t}^{i,i-1}}$ and $N' = (k-\ell)N + \ell'$.

For a $k$-uniform edge $\{ x_1, \dots, x_k\}$, we define the \emph{rational reduction}, denoted $\underline{r}(x_1,\dots,x_k)$, to be the the $i$-uniform edge $\{ \lceil x_1/(k-\ell)\rceil,  \lceil x_{(k-\ell)+1}/(k-\ell)\rceil, \dots, \lceil x_{(i-1)(k-\ell)+1}/(k-\ell)\rceil\}$.
For an $i$-uniform edge $\{x_1,\dots, x_i\}$, the \emph{canonical preimage}, denoted $\underline{r}^{-1}(x_1,\dots,x_i)$, is defined as
\[
	\underline{r}^{-1}(x_1,\dots,x_i) = \left[\bigcup_{j=1}^{i-1} \bigcup_{a=1}^{k-\ell} \{ (k-\ell)(x_j-1) + a\}\right] \cup \left[\bigcup_{a=1}^{\ell'} \{ (k-\ell)(x_i-1) + a\}\right].
\]
Observe that  $(i-1)(k-\ell) + \ell' = k$ and hence $\underline{r}^{-1}(x_1,\dots,x_i)$ has $k$ ordered elements.
Finally, note that $\underline{r}$ sends $k$-uniform edges from $K_{N'}^k$ to $i$-uniform edges in $K_N^i$ and $\underline{r}^{-1}$ sends $i$-uniform edges from $K_N^i$ to $k$-uniform edges in $K_{N'}^k$.

Let $N = \OR{}{P_{e_1}^{i,i-1}, \dots, P_{e_t}^{i,i-1}}$ and $N' = (k-\ell)N + \ell'$.

\emph{Lower Bound.}
There exists a $t$-coloring $c : E(K_{N-1}^i) \to [t]$ of $K_{N-1}^i$ that avoids a $j$-colored copy of $P_{e_j}^{i,i-1}$ for each $j \in [t]$.
Define a coloring $c' : E(K_{N'-1}^k) \to [t]$ by $c'(x_1, \dots, x_k) = c(\underline{r}(x_1,\dots,x_k))$.
Suppose that there is a color $j$ and a list $x_1 < \cdots < x_{m}$ of vertices such that there is a $j$-colored copy of $P_{e_j}^{k,\ell}$ in $c'$ on the vertices $x_1,\dots,x_m$.
Then, for each $k$-uniform edge $\{x_p,\dots,x_{p+k-1}\}$ in this copy of $P_{e_j}^{k,\ell}$, the edge $\underline{r}(x_p,\dots,x_{p+k-1})$ has color $j$ in $c$.
Also, for two consecutive edges $\{x_p,\dots,x_{p+k-1}\}$ and $\{x_{p+\ell}, \dots, x_{p+k+\ell-1}\}$ the rational reductions $\underline{r}(x_p,\dots,x_{p+k-1})$ and $\underline{r}(x_{p+\ell}, \dots, x_{p+k+\ell-1})$ intersect in $i-1$ vertices.
Thus, the $e_j$ edges given by the rational reductions form a $j$-colored copy of $P_{e_j}^{i,i-1}$, a contradiction.
Therefore, $c'$ avoids a $j$-colored copy of $P_{e_j}^{k,\ell}$ and hence $\OR{}{P_{e_1}^{k,\ell},\dots,P_{e_t}^{k,\ell}} \geq N'$.

\emph{Upper Bound\footnote{The authors thank Josef Cibulka for providing the translation of colorings in this direction.}.} 
Let $c' : E(K_{N'}^k) \to [t]$ be a $t$-coloring of $K_{N'}^k$.
Define a $t$-coloring $c : E(K_N^i) \to [t]$ of $K_N^i$ as $c(\{x_1, x_2, \dots, x_i\}) = c'( \underline{r}^{-1}(x_1, \dots, x_i))$.
By the definition of $N$, there exists a $j$-colored copy of $P_{e_j}^{i,i-1}$ on vertices $x_1,\dots,x_{m}$ for some $j \in [t]$.
For each $i$-uniform edge $\{x_q, \dots, x_{q+i-1}\}$ in this copy of $P_{e_j}^{i,i-1}$, the $k$-uniform edge $\underline{r}^{-1}(x_q,\dots,x_{q+i-1})$ also has the color $j$ with respect to $c'$.
Further, for two consecutive $i$-uniform edges $\{x_q, \dots, x_{q+i-1}\}$ and $\{x_{q+1},\dots,x_{q+i}\}$ in this copy of $P_{e_j}^{i,i-1}$, the $k$-uniform edges $\underline{r}^{-1}(x_q,\dots,x_{q+i-1})$ and  $\underline{r}^{-1}(x_{q+1},\dots,x_{q+i})$ intersect in exactly $\ell$ vertices.
Therefore, there is a $j$-colored copy of $P_{e_j}^{k,\ell}$ with respect to the coloring $c'$ and therefore $\OR{}{P_{e_1}^{k,\ell},\dots, P_{e_t}^{k,\ell}} \leq N'$.
\end{proof}

Now that we have determined the ordered Ramsey number for a particularly ``nice'' ordering of a $(k,\ell)$-path, it is natural to ask for general bounds on $\OR{t}{P_e^{k,\ell}}$ where the vertices of $P_e^{k,\ell}$ are ordered arbitrarily. In order to simplify that statement of the next lemma and theorem, we deviate slightly from our standard notation and use $P_p$ instead of $P_{p-1}^{2,1}$ to denote the $2$-uniform path on $p$ vertices.
The case for $t = 2$ was originally proven by Cibulka, Gao, Kr\v{c}\'al, Valla, and Valtr~\cite[Theorem 6]{CGKVV}; we include the full proof for the sake of completeness.

\begin{lemma}\label{lma:completepath}
Let $n$ and $p$ be positive integers, and let $P_{2^p}$ be any ordering of the $2$-uniform ordered path on $2^p$ vertices.
Then
\[
\OR{}{K_{2^n},\overbrace{P_{2^p},\ldots,P_{2^p}}^{t-1}}\leq 2^{{1\over p}\left((p+1)^{t-1}(np-1)+1\right)}.
\]
\end{lemma}

\begin{proof}
We prove by first showing that the theorem holds for all $n$ when $t=2$, and then continue by induction on $t$. For $n=1$ and $t=2$, we see that $\OR{}{K_2,P_{2^p}}=2^p=2^{{1\over p}\left((p+1)(p-1)+1\right)}$.

Let $V(P_{2^p})=\{v_1,\ldots,v_{2^p}\}$ with indices $i_1,\dots,i_{2^p}$ defined such that the ordering on $V(P_{2^p})$ is $v_{i_1}<\cdots<v_{i_{2^p}}$.

Consider a $2$-coloring $c$ of $E(K_N)$ where $N=2^{(p+1)n-1}=2^{p}M$ with $M=2^{(p+1)(p-1)}$. Let $V_1,\ldots,V_{2^p}$ be intervals partitioning $[N]$ with $|V_i|=M$ and $\max V_i<\min V_{i+1}$.
As per the ordering of $V(P_{2^p})$, let $U_j=V_{i_j}$. 
Thus, any path $(u_1,\ldots,u_{2^p})$ with $u_j\in U_j$ is a copy of $P_{2^p}$. 

For $j\in[2^p]$ define $A_j$ to be the set of vertices $v$ in $U_j$ such that there exist $u_k \in U_k$ for $k \in [j-1]$ such that $c(u_1,u_2)=c(u_2,u_3)=\cdots=c(u_{j-1},v)=2$.
Notice that $A_1=U_1$ and $A_{2^p}=\emptyset$ by the assumption that $c$ avoids $P_{2^p}$ in color $2$. 
Let $I$ be the largest integer such that $|A_I|\geq M/2$; thus, let $A=A_I$ and $B=U_{I+1}\setminus A_{I+1}$. 
Note that $|B|\geq M/2$ and the bipartite graph induced by $(A,B)$ has no edges of color $2$. 

Observe that $M/2=2^{(e+1)(n-1)-1}\geq\OR{}{K_{2^{n-1}},P_{2^p}}$ by the induction hypothesis on $n$. 
Therefore, $A$ or $B$ has a $P_{2^p}$ in color $2$ or both have a copy of $K_{2^{n-1}}$ in color $1$. 
If the former is true, we are done, so suppose the latter holds. Therefore, $A\cup B$ has a $K_{2^n}$ in color $1$, so $\OR{}{K_{2^n},P_{2^p}}\leq 2^{(p+1)n-1}$.

Now, suppose that $t>2$ and consider a $t$-coloring, $c$, of $E(K_N)$ for $N=2^{{1\over p}\left((p+1)^{t-1}(np-1)+1\right)}$. Realizing that ${(p+1)^{t-1}(np-1)+1\over p}=(p+1){(p+1)^{t-2}(np-1)+1\over p}-1$, we find through the $t=2$ case that
\[
N\geq\OR{}{K_{2^{{1\over p}\left((p+1)^{t-2}(np-1)+1\right)}},P_{2^p}}.
\]
Thus, $c$ either has a $P_{2^p}$ in color $t$ or a $K_{2^{{1\over p}\left((p+1)^{t-2}(np-1)+1\right)}}$ which is void of color $t$. If the former holds, then we are done, so suppose the latter holds. 
By the induction hypothesis on $t$, 
\[
2^{{1\over p}\left((p+1)^{t-2}(np-1)+1\right)}\geq\OR{}{K_{2^n},\overbrace{P_{2^p},\ldots,P_{2^p}}^{t-2}};
\]
therefore, we either have a $K_{2^n}$ in color $1$ or a $P_{2^p}$ in some color $j \in \{2,\ldots,t-1\}$.
\end{proof}

Lemma~\ref{lma:completepath} immediately implies the following theorem.

\begin{theorem}\label{arbpath}
Let $P_p$ be any ordered $2$-uniform path on $p$ vertices, then
\[
\OR{t}{P_p}\leq 2^{{1\over \lceil\lg p\rceil}\left((\lceil\lg p\rceil+1)^{t-1}(\lceil\lg p\rceil^2-1)+1\right)}=2^{O\left(\lg^t p\right)}.
\]
\end{theorem}

As a means to a lower bound on this value, Conlon, Fox, Lee and Sudakov~\cite{CFLS} provided the following lower bound on the ordered Ramsey number of a randomly-ordered 2-uniform matching, which was also proved in a weaker form by Balko, Cibulka, Kr\'{a}l and Kyn\u{c}l \cite{BK}.

\begin{theorem}[Conlon, Fox, Lee and Sudakov \protect{\cite[Theorems 2.3]{CFLS}}]\label{conlonthm}
There exists a positive constant $c$, such that if $M$ is a randomly-ordered matching on $e$ edges, then asymptotically almost surely,
\[
\OR{2}{M}\geq (2e)^{c\log(2e)/\log\log(2e)}.
\]
\end{theorem}

Since $P_p$ contains a matching of size $\lfloor p/2\rfloor$, we see that almost every ordering of $P_p$ yields $\OR{2}{P_p}\geq 2^{\Omega(\lg^2 p/\lg\lg p)}$. 
Hence, Theorem \ref{arbpath} is fairly tight when $t=2$.
Therefore, for almost every ordering of $P_p$, $\OR{t}{P_p}$ grows as a quasi-polynomial in $p$ for a fixed $t$ and possibly double-exponentially in $t$ for a fixed $p$. 
Comparatively, for the standard ordering of $P_p$, $\OR{t}{P_p}$ grows polynomially in $p$ and exponentially in $t$.

\section{Ordered Ramsey Numbers of $k$-Uniform Matchings}\label{matching}

Recall that the ordered path $P_e^{k,0}$ has disjoint edges, and therefore is a matching.
The proof of Theorem~\ref{mainthmq} holds for $\ell = 0$, but instead we will consider a more general class of ordered matchings.

For a fixed $0\leq r\leq k$ and positive integer $e$, the \emph{$(k,r)$-nested matching on $e$ edges} is the ordered graph $M_{e}^{k,r}$ defined iteratively as: $E(M_1^{k,r})$ consists of one edge $A_1 = [k]$, and $E(M_{e+1}^{k,r})$ consists of the edges in $E(M_{e}^{k,r})$ and an edge $A_{e+1}$ consisting of the $r$ least integers greater than $\max V(M_{e}^{k,r})$ and  the $k-r$ greatest integers less than $\min V(M_e^{k,r})$.
We say $(k,r)$ is the \emph{nesting pattern} of $M_e^{k,r}$.
Note that $M_{e}^{k,r}$ is isomorphic to $M_{e}^{k,k-r}$ when the ordering is reversed, and $M_{e}^{k,0}\cong M_e^{k,k}\cong P_e^{k,0}$.


In \cite{CL}, Cockayne and Lorimer show that for integers $e_1\geq \cdots\geq e_t$, if $M_i$ is a $2$-uniform matching on $e_i$ edges, then 
\[
\R{}{M_1,\ldots,M_t}=e_{1}+1+\sum_{i=1}^t(e_i-1).
\]
This value is not far from the value of the ordered Ramsey number for 2-uniform nested matchings.
The following lemma presents a lower bound on the ordered Ramsey number of $t$ $k$-uniform nested matchings, even if the nesting patterns differ among the matchings.

\begin{lemma}\label{lma:matchlower}
For positive integers $e_1,\ldots,e_t$ and $r_1,\ldots,r_t\in\{0,\ldots,k\}$,
\[
\OR{}{M_{e_1}^{k,r_1},\ldots,M_{e_t}^{k,r_t}}\geq k\left(1+\sum_{i=1}^t (e_i-1)\right).
\]
\end{lemma}

\begin{proof}
Let $N=k\left(1+\sum_{i=1}^t (e_i-1)\right)-1$. 
Let $L_1, \ldots, L_t, R_1,\ldots,R_t$ be intervals partitioning $[N]$, with $L_1=R_1$, such that for $i\in\{1,\ldots,t-1\}$, $\max L_{i+1} < \min L_{i}$  and $\max R_i<\min R_{i+1}$. 
Further, let $|L_1|=ke_1-1$, and for $i\in\{2,\ldots,t\}$ let $|L_i|=(k-r_i)(e_i-1)$ and $|R_i|=r_i(e_i-1)$. 
For an edge $X\in{[n]\choose k}$, let $c(X)=\max\{i : X \cap (L_i\cup R_i)\neq\emptyset\}$. 
The interval $L_1$ is too small for  $c$ to contain a copy of $M_{e_1}^{k,r_1}$ in color $1$. 

Suppose that $c$ contained a copy of $M_{e_i}^{k,r_i}$ in color $i$ for some $i\in\{2,\ldots,t\}$. 
If $r_i = k$, then $L_i = \emptyset$ and $|R_i| = k(e_i-1)$; therefore some edge of $M_{e_i}^{k,r_i}$ does not intersect $R_i$ and hence does not have color $i$.
The case $r_i = 0$ is similar, except $|L_i| = k(e_i-1)$ and $R_i = \emptyset$.

Now suppose $1 \leq r_i < k$.
Let $p_1,\ldots,p_{e_i}$ be the minimum vertices of the edges of $M_{e_i}^{k,r_i}$ and $q_1,\ldots,q_{e_i}$ be the set of maximum vertices, hence $p_1<p_2<\cdots<p_{e_i}<q_{e_i}<\cdots<q_1$. 
In fact, $p_m+k-r_i<p_{m+1}$ and $q_m-r_i>q_{m+1}$ for $m=1,\ldots,e_i-1$. 
Since each edge receives color $i$, either $p_m\in L_i$ or $q_m\in R_i$ for all $m$. 
However, because $|L_i|=(k-r_i)(e_i-1)$ and $|R_i|=r_i(e_i-1)$, it must be the case that $p_{e_i}\notin L_i$ and $q_{e_i}\notin R_i$. 
Therefore, $c$ avoids $M_{e_i}^{k,r_i}$ for all $i$. 
\end{proof}

When all nesting patterns are the same, the bound from Lemma~\ref{lma:matchlower} is sharp.

\begin{theorem}\label{cononmatch}
For positive integers $e_1,\ldots,e_t$, and $0\leq r\leq k$,
\[
\OR{}{M_{e_1}^{k,r},\ldots,M_{e_t}^{k,r}}= k\left(1+\sum_{i=1}^t (e_i-1)\right).
\]
\end{theorem}

\begin{proof}
The lower bound follows from Lemma~\ref{lma:matchlower}.
We prove the upper bound by induction on $\sum_{i=1}^t e_i$. 
If $\sum_{i=1}^t e_i=t$, then $e_i=1$ for all $i$, so $\OR{}{M_{e_1}^{k,r},\ldots,M_{e_t}^{k,r}}=k$, and the claim holds.

Suppose that $\sum_{i=1}^t e_i>t$ and let $c$ be a $t$-coloring of $E(K_N^k)$ where $N=k\left(1+\sum_{i=1}^t (e_i-1)\right)$. 
Suppose that $c(\{1,\ldots,r\} \cup \{N-k+r+1,\ldots,N\}) = j$ for some $j \in [t]$. 
Let $G$ be the graph given by deleting the vertices in $\{1,\ldots,r\} \cup \{N-k+r+1, \ldots, N\}$ from $K_N^k$.
Let $e_j'=e_j-1$ and $e_i'=e_i$ for $i\neq j$. 
Notice that $G\cong K_{N-k}^k$ and $N-k=k\left(1+\sum_{i=1}^t (e_i'-1)\right)$. Therefore, since $\sum_{i=1}^t e_i'=\sum_{i=1}^t e_i-1$, the induction hypothesis implies that $G$ contains an $i$-colored copy of $M_{e_i'}^{k,r_i}$ for some $i$.
Since $e_i' = e_i$ when $i \neq j$, we have $i = j$.
Then the $j$-colored copy of $M_{e_j-1}^{k,r_j}$ along with the edge $\{1,\ldots,r\} \cup \{N-k+r+1,\ldots,N\}$ is a $j$-colored copy of $M_{e_j}^{k,r_j}$.
\end{proof}

Notice that $i(k,0)=1$ and $|Q_1(e_1,\ldots,e_t)|=\sum_{i=1}^t(e_i-1)$; thus, the $r=0$ case of Theorem \ref{cononmatch} agrees with the bound in Theorem \ref{mainthmq} using $\ell = 0$. 
Interestingly, as opposed to the large discrepancy between the ordered and ordinary Ramsey numbers of paths, we see that $\OR{t}{M_{e}^{2,r}}\leq 2\R{t}{M_{e}^{2,r}}$. 
However, this trend does not continue when the ordering of the matching is not nested as in $M_{e}^{k,r}$. 
Likely $M_{e}^{k,r}$ minimizes the ordered Ramsey number $\OR{t}{M}$ among all orderings of $k$-uniform matchings $M$ on $e$ edges.

Conlon, Fox, Lee and Sudakov \cite{CFLS} explore the ordered Ramsey numbers of $2$-uniform matchings. 

\begin{theorem}[Conlon, Fox, Lee and Sudakov \protect{\cite{CFLS}}]\label{thm:cfls}
Let $M_2,\dots,M_t$ be ordered $2$-uniform matchings, and let $p \geq 2$.
Then $\OR{}{K_{p}, M_2,\dots,M_t} \leq \OR{}{M_2,\dots,M_t}^{\lceil \lg p\rceil}$.
Therefore, for an ordered $2$-uniform matching $M$ with $e$ edges, $\OR{t}{M}\leq (2e)^{\lceil\lg(2e)\rceil^{t-1}} \leq 2^{\lceil\lg(2e)\rceil^t}.$
\end{theorem}

Compare the upper bound here with the lower bound from Theorem~\ref{conlonthm}, showing that this upper bound is nearly tight.
In terms of $e$, the bound above is quasi-polynomial, but in terms of $t$ the bound is doubly-exponential.

Define the $k$-uniform graph $G_{s}^k$ iteratively on $s$ as follows: let $G_{0}^k$ consist of a single vertex, and for $s\geq 1$, let $G_{s}^k$ consist of $k$ disjoint, consecutive copies of $G_{s-1}^k$, and introduce every $k$-uniform edge consisting of exactly one vertex from each copy. 
Notice that $G_{s}^2=K_{2^s}$. 

The above definition of $G_s^k$ uses a ``concatenation'' step to glue $k$ copies of $G_{s-1}^k$ to form $G_s^k$.
We now state an equivalent definition, which we refer to as the ``blow-up'' construction of $G_s^k$, that uses an ``expansion'' step that is key to the proof of Lemma~\ref{thm:nesting}.
Let $V(G_{s-1}^k)=\{x_1,\ldots,x_{k^{s-1}}\}$ with $x_i < x_{i'}$ if and only if $i < i'$. 
Duplicate each vertex $x_i$ $k$ times to form a list of vertices $x_i^{(1)}, \dots, x_i^{(k)}$.
Two vertices $x_{i}^{(j)}$ and $x_{i'}^{(j')}$ are ordered as $x_i^{(j)} < x_{i'}^{(j')}$ if $i < i'$, or $i = i'$ and $j < j'$.
The graph $G_s^k$ has vertex set $\{ x_{i}^{(j)} : i \in [k^{s-1}], j \in [k]\}$ and the edges of $G_s^k$ are of the form $\{x_i^{(1)},\dots,x_i^{(k)}\}$ for $i \in [k^{s-1}]$ or $\{x_{i_1}^{(j_1)}, \dots, x_{i_k}^{(j_k)}\}$ for every edge $\{x_{i_1},\dots,x_{i_k}\}$ in $G_{s-1}^k$ and any tuple $(j_1,\dots,j_k) \in [k]^k$.

Using the graph $G_{s}^k$, we attain a bound on the $t$-color ordered Ramsey numbers of certain ``nice'' orderings of $k$-uniform matchings. 
This bound is a generalization of Theorem~\ref{thm:cfls}, where $G_s^k$ replaces the complete graph.

\begin{lemma}\label{lem:iterated}
Let $M_2,\ldots,M_t$ be any $k$-uniform ordered matchings and $s \geq 0$.
Then 
\[
\OR{}{G_{s}^k,M_2,\ldots,M_t}\leq  \OR{}{M_2,\ldots,M_t}^s.
\]
\end{lemma}

\begin{proof}
We prove by induction on $s$. 
When $s = 0$, the graph $G_{0}^k$ consists of a single vertex, and hence every coloring of $K_1^k$ contains a copy of $G_{s}^k$ in every color.

Suppose that $s > 0$ and let $r = \OR{}{M_2,\dots,M_t}$.
Let $c$ be a $t$-coloring of $K_{r^{s}}^k$ that avoids a $j$-colored copy of $M_j$ for each $j \in \{2,\dots,t\}$ and avoids a 1-colored copy of $G_{s}^k$.
Let $V_1,\ldots,V_r$ be equal-sized intervals partitioning $[r^{s}]$ such that $\max V_i < \min V_{i+1}$ for $i \in [r-1]$.
By the induction hypothesis, restricting $c$ to $V_i$ yields either a copy of $G_{(s-1)}^k$ in color $1$ or a $j$-colored copy of $M_j$ for some $j\in\{2,\ldots,t\}$.
Since $c$ contains no $j$-colored copy of $M_j$, each $V_i$ contains a copy of $G_{(s-1)}^k$. 
Since $c$ avoids $G_{s}^k$, then for any indices $1\leq i_1<\cdots< i_k\leq r$ there must be $x_{i_j} \in V_{i_j}$ such that $c(x_{i_1},\ldots,x_{i_k})\neq 1$. 
Define a coloring of $E(K_r^k)$ by letting $c'(v_{i_1},\ldots,v_{i_k})$ be any color in $\{c(x_{i_1},\ldots,x_{i_k}):x_{i_j}\in V_{i_j}\}\setminus\{1\}$. 
By the definition of $r$, $c'$ contains an $j$-colored copy of $M_j$ for some $j\in \{2,\ldots,t\}$ and therefore $c$ also contains a $j$-colored copy of $M_j$; a contradiction.
\end{proof}

Let $M$ be an ordered $k$-uniform matching on vertex set $[ke]$.
We say that $M$ is \emph{$k$-nestable} if there exist disjoint intervals $I_1, \dots, I_k$, some of which may be empty or degenerate, spanning $[ke]$ such that $1 \in I_1, ke \in I_k$, where each edge in $M$ either is contained in some interval $I_j$ or spans all intervals $I_1, \dots, I_k$, and for each $j \in [k]$ the edges contained within $I_j$ form a matching, denoted $M_j$, that is either $k$-nestable or empty.
A set of intervals $I_1, \dots, I_k$ satisfying these properties is a \emph{$k$-nesting} of $M$. Notice that every matching contained as a subgraph of $G_s^k$ for some $s$ must be $k$-nestable; in particular, every $2$-uniform matching is $2$-nestable as $G_s^2\cong K_{2^s}$. The following lemma provides the converse to this observation.

\begin{lemma}\label{lem:nesting}
If $M$ is a $k$-nestable ordered matching with $e$ edges for $k\geq 3$, then $M$ is contained within $G_{2e-1}^k$.
\end{lemma}

\begin{proof}
We prove by induction on $e$. The statement is trivial when $e = 1$ as both $M$ and $G_1^k$ are a single $k$-uniform edges.
Suppose the ordered $k$-uniform matching $M$ has vertex set $[ke]$ for $e\geq 2$.

Let $I_1, \dots, I_k$ be a $k$-nesting of $M$, and let $M_1, \dots, M_k$ be the matchings induced by the edges within each interval.
For $j \in [k]$, let $e_j$ be the number of edges in the matching $M_j$. Define $M'$ to be the matching $M - \bigcup_{j=1}^k M_j$. Since $e_j < e$, by the inductive hypothesis, there exist order-preserving graph embeddings $\pi_j : V(M_j) \to V(G_{2e_j-1}^k)$ from $M_j$ to a subgraph within $G_{2e_j-1}^k$. 

If $M'$ happens to be empty, for $v\in V(M_j)$ define $\pi'(v)$ to be the copy of $\pi_j(v)$ in the first copy of $G_{2e_j-1}^k$ contained within $G_{2\max_j e_j -1}^k$. Further, define $\pi''(v)$ to be the copy of $\pi'(v)$ in the $j$th copy of $G_{2\max_j e_j-1}^k$ contained within $G_{2\max_j e_j}^k$. It is readily seen that $\pi''$ is an embedding of $M$ into $G_{2\max_j e_j}^k$. Because $e>\max_j e_j$, the claim follows.

Now suppose that $M'$ is nonempty, and let $e' = (e - \sum_{j=1}^k e_j) + \max_j e_j$. Because $M'$ is nonempty, $e'>\max_j e_j$. We will show that $M$ is contained within $G_{2e'-1}^k$. We begin by embedding $\bigcup_{j=1}^k M_j$ into $G_{2e'-1}^k$ using the embeddings $\pi_1,\dots, \pi_k$.
This comes in two steps: first the embedding of $M_j$ is ``expanded'' into $G_{2e'-2}^k$ by using the blow-up construction of $G_s^k$, then the $k$ embeddings into $G_{2e'-2}^k$ are ``concatenated'' to allow for an embedding of $M$ into $G_{2e'-1}^k$.
Let $v$ be a vertex in $M_j$.
Since $\pi_j(v) \in V(G_{2e_j-1}^k)$, we must convert $\pi_j(v)$ to a vertex in $G_{2e'-2}^k$.
Let $\ell(v)$ be the number of vertices in $V(M') \cap I_j$ less than $v$.
There can be at most $|E(M')|$ such vertices, so $0 \leq \ell(v) \leq |E(M')| \leq (e'-e_j) \leq k^{2(e'-e_j)-2}$.
Thus, there exists a $k$-ary representation of $\ell(v)$ as $\sum_{i=0}^{2(e'-e_j)-3} a_ik^i$ for nonnegative integers $a_0, \dots, a_{2(e'-e_j)-3}$ satisfying $0 \leq a_i < k$.
Define a list $x_{0}$, $x_{1}$, $\dots$, $x_{2(e'-e_j)-2}$ iteratively as $x_{0} = \pi_j(v)$ and $x_{i+1} = x_i^{(a_i)}$, where $x_i$ is a vertex in $G_{2e_j+i-1}^k$ and $x_i^{(a)}$ is the $a$th copy of $x_i$ in $G_{2e_j+i}^k$ as in the blow-up construction of $G_s^k$.
Let $\pi'(v) = x_{2(e'-e_j)-2}$, which is a vertex in  $G_{2e'-2}^k$.

Observe that for two consecutive vertices $u<v$ in $M_j$, there are at least $(k-1)^{2(e'-e_j)-2}$ vertices between $\pi'(u)$ and $\pi'(v)$ in $G_{2e'-2}^k$ because $\ell(v)\geq\ell(u)$, and that $|V(M')\cap[u,v]|=\ell(u)-\ell(v) \leq e'-e_j\leq (k-1)^{2(e'-e_j)-2}$ because $k\geq 3$.
Also note that if $u=\min V(M_j)$, then there are exactly $\ell(u)$ vertices in $G_{2e'-2}^k$ less than $\pi'(u)$, and if $v=\max V(M_j)$, then there are at least $|E(M')|-\ell(v)$ vertices in $G_{2e'-2}^k$ greater than $\pi'(v)$.
Now  for $v \in V(M_j)$, define $\pi''(v)$ to be the copy of $\pi'(v)$ in the $j$th copy of $G_{2e'-2}^k$ within $G_{2e'-1}^k$.

We now select vertices in $G_{2e'-1}^k$ to embed the vertices of $M'$.
Consider an interval $I_j$, let $v_{\min}$ be the least vertex in $M_j$, and let $v_{\max}$ be the greatest vertex in $M_j$.
There are $\ell = \ell(v_{\min})$ vertices $u_1,\dots,u_\ell$ of $M'$ in $I_j$ that precede $v_{\min}$, and the same number of vertices $x_1,\dots,x_\ell$ in the $j$th copy of $G_{2e'-2}^k$ less than $\pi''(v_{\min})$; hence we define $\pi''(u_i) = x_i$ for $i \in [\ell]$.
For two consecutive vertices $u\leq v$ of $M_j$, there are $m=\ell(v)-\ell(u)$ vertices $u_1,\dots,u_m$ of $M'$ between $u$ and $v$, and at least $(k-1)^{2(e'-e_j)-2}\geq\ell(v)-\ell(u)$ vertices in the $j$th copy of $G_{2e'-2}^k$ between $\pi''(u)$ and $\pi''(v)$. Therefore, we can select the vertices $\pi''(u_1),\dots,\pi''(u_m)$ in order.
Finally, there are $n=|E(M')|-\ell(v_{\max})$ vertices $u_1,\dots, u_n$ of $M'$ in $I_j$ that are greater than $v_{\max}$, and there are at least $|E(M')|-\ell(v_{\max})$ vertices in the $j$th copy of $G_{2e'-2}^k$ greater than $\pi''(v_{\max})$, so we can select the vertices $\pi''(u_1), \dots, \pi''(u_n)$ in order.
The resulting injection $\pi'' : V(M) \to G_{2e'-1}^k$ is an embedding of $M$ into $G_{2e'-1}^k$.
\end{proof}

Note that in the proof of Lemma~\ref{lem:nesting}, the ``expansion'' step takes $2(e'-e_j)-1$ iterations. 
In the case of one of the standard nesting matchings $M_e^{k,r}$, this is exactly one iteration.
Thus, even for a matching $M_e^{k,r}$ where the ordered Ramsey number is small, it is not possibly to embed $M_e^{k,r}$ into $G_s^k$ for any $s<2e-1$ whenever $1\leq r\leq k-1$. When a $k$-nesting contains two nonempty matchings $M_j$ and $M_{j'}$, or when there are multiple edges in $M'$, the iterative process given above may require fewer than $2e-1$ steps.
However, it does appear that $\Omega(e)$ steps are required for most $k$-nested matchings on $e$ edges, as most of the edges will likely live in $M'$.

The following theorem follows from Lemmas~\ref{lem:iterated} and \ref{lem:nesting} and the fact that $\OR{1}{M} = ek$ if $M$ is a $k$-uniform ordered matching with $e$ edges.

\begin{theorem}\label{thm:nesting}
Let $k \geq 3$ and $e \geq 2$.
If $M$ is a $k$-nestable ordered matching with $e$ edges, then $\OR{t}{M} \leq (ek)^{(2e-1)^{t-1}} = 2^{(2e-1)^{t-1}\lg(ek)}$.
\end{theorem}

This extends the previous bound on 2-uniform matchings~\cite{CFLS}.
While the bound remains doubly-exponential in terms of $t$, the bound has increased from quasi-polynomial to exponential in terms of $e$. Most notably, this bound is only polynomial in $k$.

Notice that for these ``nice'' orderings of a $k$-uniform matching on $e$ edges, the bound on the ordered Ramsey number $\OR{t}{M}$ is only slightly larger than the ordered Ramsey number $\OR{t}{P_e^{k,\ell}}$ of the naturally-ordered $(k,\ell)$-path on $e$ edges when $i(k,\ell) = 3$.

We say that a $k$-uniform ordered matching $M$ is \emph{simply interlacing} if for any pair of distinct edges $A, B$ in $M$, where $A = \{ a_1 < a_2 < \cdots < a_k\}$ and $B = \{ b_1 < b_2 < \cdots < b_k\}$ either $a_i$ and $b_i$ are consecutive in $A \cup B$ for each $i$ or there is some $i$ where $a_i < b_1 < b_k < a_{i+1}$ (where $a_0 = -\infty$ and $a_{k+1} = +\infty$). If the former holds, we say that $A$ and $B$ \emph{interlace}, and if the latter holds, we say that $A$ and $B$ \emph{nest}. Notice that every $2$-uniform matching is simply interlacing.

\begin{cor}\label{cor:simplyintersecting}
If $k \geq 3$, $e \geq 2$, and $M$ is a simply-interlacing $k$-uniform ordered matching with $e$ edges, then $M$ is $k$-nestable; hence $\OR{t}{M} \leq (ek)^{(2e-1)^{t-1}} = 2^{(2e-1)^{t-1}\lg(ek)}$.
\end{cor}

\begin{proof}
By Theorem~\ref{thm:nesting}, it suffices to show that $M$ is $k$-nestable.
Define a relation on the edges of $M$ by $A \prec B$ if $b_i<a_1<a_k<b_{i+1}$ for some $0\leq i\leq k-1$, where $A=\{a_1<\dots<a_k\}$ and $B=\{b_1<\dots<b_k\}$ (again under the convention that $b_0=-\infty$). While $\prec$ is not a partial order (as transitivity fails), it does admit maximal elements. Let $A_1, \dots, A_p$ be the edges of $M$ that are either maximal with respect to $\prec$ or interlace with some maximal edge. Therefore, $A_i$ and $A_{i'}$ interlace. We refer to these edges as \emph{spanning edges}.

For each $i \in [p]$, label the vertices in $A_i$ as $A_i = \{ a_{i,1} < \cdots < a_{i,k}\}$; also let $a_{i,0} = -\infty$ and $a_{i,k+1} = +\infty$.
Observe that for each $j \in [k-1]$, we have $\max_{i \in [p]} a_{i,j} < \min_{i\in [p]} a_{i,j+1}$, as otherwise there is a pair of edge $A_i$ and $A_{i'}$ where $a_{i,j} > a_{i',j+1}$ and hence $a_{i,j}$ and $a_{i',j}$ are not consecutive in $A_i\cup A_{i'}$.
Therefore, we can define disjoint intervals $I_1, \dots, I_k$ such that $I_j = [\min_{i\in[p]} a_{i,j}, \max_{i\in[p]}a_{i,j}]$. 
These intervals do not necessarily span $V(M)$, but we will expand them to include vertices not in $A_1, \dots, A_p$.

For a non-spanning edge $B$ in $M$, there is at least one edge $A_i$ where $B \prec A_i$.
Therefore, there exists a $j \in \{0,\dots,k-1\}$ such that $a_{i,j} < \min B < \max B < a_{i,j+1}$.
Observe that since $k \geq 3$, for any $i' \in [p]$ the edge $B$ is comparable to $A_{i'}$ since there is some $a_{i',j'}$ not in the interval $[a_{i,j},a_{i,j+1}]$.
While it may not be the case that $B \prec A_{i'}$, it is true that for every $i' \in [p]$ and $a_{i', j + c_{i'}} < \min B < \max B < a_{i', j+c_{i'}+1}$ for some $c_{i'} \in \{-1,0,+1\}$, as $A_{i'}\prec B$ only when $a_{i',k}<\min B$.
Therefore, let $j_B$ be the minimum integer satisfying $j_B \geq 1$ and $j_B \geq j + c_{i'}$ for each $i' \in [p]$.

If  $B, B'$ are two non-spanning edges in $M$ and $j_B < j_{B'}$, then $\max B < a_{i,j_B+1}$ for all $i \in [p]$ and $a_{i', j_{B'}} < \min B'$ for some $i' \in [p]$.
Then $\max B < a_{i',j_B+1} < \min B'$.
Therefore, if for every non-spanning edge $B$ in $M$ we minimally extend the interval $I_{j_B}$ to contain the edge $B$, the intervals $I_1, \dots, I_k$ will always be disjoint.

Note that the matching $M_j$ given by the edges entirely within the interval $I_j$ is a simply-interlacing $k$-uniform ordered matching and hence is $k$-nestable by an inductive argument.
Therefore, the intervals $I_1,\dots,I_k$ form a $k$-nesting of $M$.
\end{proof}

We conclude by noting that Lemma~\ref{lem:iterated} will not apply to most ordered $k$-uniform matchings for $k \geq 3$.
For $k \geq 4$, let $A$ and $B$ be defined as
\[
	A = \{ 1, \dots, \lfloor k/2\rfloor \} \cup \{ k+1,\dots, k + \lceil k/2\rceil\}, \quad B = \{ \lfloor k/2\rfloor + 1, \dots, k\} \cup \{ k + \lceil k/2\rceil, \dots, 2k\}.
\]
Observe that the ordered matching with edges $A$ and $B$ is not $k$-nestable.
While every ordered $3$-uniform matching on two edges is $3$-nestable, there exists an ordered $3$-uniform matching that is not $3$-nestable.
A randomly-ordered matching contains these configurations with high probability, so the bound of Theorem~\ref{thm:nesting} does not apply to most ordered matchings.

\section{Future Directions}\label{future}

Our investigation into arbitrarily-ordered $k$-uniform matchings provides upper bounds that are similar to the previous bounds in the 2-uniform case.
Extending the techniques from 2-uniform matchings comes at the cost that it does not apply to all $k$-uniform ordered matchings, but they do provide bounds that are exponential and not a tower.
However, our methods do not allude to lower bounds, and hence it is unclear whether our upper bounds are tight. 

The largest question left open from our study of ordered Ramsey numbers is related to arbitrary orderings of $(k,\ell)$-paths. 
While we found upper bounds on $\OR{t}{P_e^{2,1}}$, our techniques did not easily extend to higher uniformities. 
Upper bounds on $\OR{t}{P_e^{k,\ell}}$ for arbitrary orderings of $P_e^{k,\ell}$ would be very interesting and would significantly extend our current techniques. Noticing that $\tow{k-2}{\Omega(n^2)}\leq\R{2}{K_n^k}\leq\tow{k-1}{O(n)}$ (see~\cite{EHR}), the bound for $\OR{t}{P_e^{k,k-1}}$ for the natural ordering cannot be far off a general bound for $\OR{t}{P_e^{k,k-1}}$ for an arbitrary ordering. However, $\OR{t}{P_e^{k,\ell}}$ for the natural ordering grows as a tower of height $i(k,\ell)-1$, so the upper bound for $\OR{t}{P_e^{k,\ell}}$ for an arbitrary ordering may be much larger, especially if $i(k,\ell)=2$. 
Thus, bounds on tight paths may not lead to bounds on loose paths in the same way that Theorem \ref{thm:main} draws this connection for monotone paths.

The generalized diamond $D_r$ consists of $r$ copies of $P_2^{2,1}$ who share first and last vertices. The ordering of the intermediate vertices is unimportant as all orderings yield isomorphic graphs. Balko, Cibulka, Kr\'{a}l, and Kyn\u{c}l~\cite{BK} determined that $\OR{2}{D_2}=11$. We would like to determine, asymptotically or otherwise, the growth of $\OR{t}{D_r}$ in terms of $r$. While the study of monotone paths explains what happens when a graph gets ``longer,'' the study of the generalized diamond will yield a better understanding of what happens when a graph gets ``wider.''

The natural extension of $D_r$ to higher uniformities, $D_r^{k,\ell}$, consists of $r$ copies of $P_2^{k,\ell}$ who share their first $k-\ell$ and last $k-\ell$ vertices. 
However, unless $\ell=1$, $D_r^{k,\ell}$ admits many nonisomorphic orderings of the intermediate vertices, none of which are essentially natural. 
Presumably, a somewhat symmetric ordering of the intermediate vertices will minimize $\OR{t}{D_r^{k,\ell}}$, but other than the fact that it is bounded below by $\OR{t}{P_2^{k,\ell}}$, it is unclear how large this number can become.

\section*{Acknowledgments}

The authors would like to thank David Conlon and Josef Cibulka for remarks that helped improve this paper.
In particular, Josef Cibulka presented the translation of colorings in the direct proof of Theorem~\ref{thm:main}.

\end{document}